\documentclass[a4paper,oneside]{amsart}
\usepackage[T1]{fontenc}
\usepackage[utf8]{inputenc}
\usepackage[all,ps,arc]{xy}  
\usepackage{amsmath,amssymb,latexsym}
\usepackage{amsfonts}
\usepackage{amsthm}
\usepackage{mathabx}
\usepackage{xcolor}
\usepackage{bbm}
\usepackage[textsize=scriptsize]{todonotes}
\usepackage[shortlabels]{enumitem}
\usepackage{thmtools,mathtools}
\usepackage{nameref,hyperref,cleveref}
\hypersetup{%
    colorlinks=true, linktocpage=true, pdfstartpage=3, pdfstartview=FitV,%
    breaklinks=true, pdfpagemode=UseNone, pageanchor=true, pdfpagemode=UseOutlines,%
    plainpages=false, bookmarksnumbered, bookmarksopen=true, bookmarksopenlevel=1,%
    hypertexnames=true, pdfhighlight=/O,
    urlcolor=webbrown, linkcolor=RoyalBlue, citecolor=webgreen, 
    pdfsubject={},%
    pdfkeywords={},%
    pdfcreator={pdfLaTeX},%
}   \definecolor{webgreen}{rgb}{0,.5,0}
\definecolor{RoyalBlue}{rgb}{0,0,.5}

\newcommand{\point}{\ensuremath{\xymatrix{A\ar@<+.6ex>[r]^(.5){\alpha}
&B\ar@<+.6ex>[l]^(.5){\beta}}}}
\newcommand{\rg}{\ensuremath{\xymatrix{A\ar@<+1ex>[r]^{\alpha}\ar@<-1ex>[r]_{\gamma}&B\ar[l]|{\beta}}}}

\newtheorem{theorem}{Theorem}[section]
\newtheorem{lemma}[theorem]{Lemma}

\newtheorem{corollary}[theorem]{Corollary}
\theoremstyle{remark}
\newtheorem{remark}[theorem]{Remark}
\newtheorem{example}[theorem]{Example}

\theoremstyle{definition}
\newtheorem{definition}[theorem]{Definition}

\newcommand{\df}{\ensuremath{\coloneqq}}
\newcommand{\sA}{\ensuremath{\mathsf A}}
\newcommand{\sB}{\ensuremath{\mathsf B}}

\newcommand{\sV}{\ensuremath{\mathsf V}}
\newcommand{\Pt}{\ensuremath{\mathsf{Pt}}}

\newcommand{\two}{\ensuremath{\mathbf{2}}}

\newcommand{\HSLat}{\ensuremath{\mathsf{HSLat}}}
\newcommand{\HAlg}{\ensuremath{\mathsf{HAlg}}}
\newcommand{\MVAlg}{\ensuremath{\mathsf{MVAlg}}}

\newcommand{\Hoops}{\ensuremath{\mathsf{Hoops}}}
\newcommand{\WHoops}{\ensuremath{\mathsf{WHoops}}}
\newcommand{\PHoops}{\ensuremath{\mathsf{PHoops}}}
\newcommand{\GHoops}{\ensuremath{\mathsf{GHoops}}}
\newcommand{\BHoops}{\ensuremath{\mathsf{BHoops}}}
\newcommand{\IHoops}{\ensuremath{\mathsf{IHoops}}}

\newcommand{\CRing}{\ensuremath{\mathsf{CRing}}}
\newcommand{\CRng}{\ensuremath{\mathsf{CRng}}}
\newcommand{\Ring}{\ensuremath{\mathsf{Ring}}}
\newcommand{\Rng}{\ensuremath{\mathsf{Rng}}}
\newcommand{\BooAlg}{\ensuremath{\mathsf{BooAlg}}}
\newcommand{\BooRng}{\ensuremath{\mathsf{BooRng}}}
\newcommand{\UAss}{\ensuremath{\mathsf{UAss}}}
\newcommand{\Ass}{\ensuremath{\mathsf{Ass}}}
\newcommand{\UCStar}{\ensuremath{\mathsf{UCStar}}}
\newcommand{\CStar}{\ensuremath{\mathsf{CStar}}}

\newcommand{\Ker}{\ensuremath{\text{Ker}}}

\newdir{|>}{{}*!/8pt/@{|}*!/3.5pt/:(1,-.2)@^{>}*!/3.5pt/:(1,+.2)@_{>}}
\newdir{ >}{{}*!/-10pt/@{>}}
\newdir{ |>}{!/10pt/{}*!/4.5pt/@{|}*:(1,-.2)@^{>}*:(1,+.2)@_{>}}

\begin{document}

\def \tm{\!\times\!}

\newenvironment{changemargin}[2]{\begin{list}{}{
\setlength{\topsep}{0pt}
\setlength{\leftmargin}{0pt}
\setlength{\rightmargin}{0pt}
\setlength{\listparindent}{\parindent}
\setlength{\itemindent}{\parindent}
\setlength{\parsep}{0pt plus 1pt}
\addtolength{\leftmargin}{#1}\addtolength{\rightmargin}{#2}
}\item}{\end{list}}

\title[Relative ideals in homological categories]{Relative ideals in homological categories\\ with an application to MV-algebras}
\author[S.~Lapenta]{Serafina Lapenta}
\address{Università di Salerno, Dipartimento di Matematica.  Piazza Renato Caccioppoli, 2 – 84084 Fisciano (SA), Italy}
\email{slapenta@unisa.it}
\author[G.~Metere]{Giuseppe Metere}
\address{Università di Milano, Via Luigi Mangiagalli, 25 – 20133 Milano, Italy}
\email{giuseppe.metere@unimi.it}
\author[L.~Spada]{Luca Spada}
\address{Università di Salerno, Dipartimento di Matematica.  Piazza Renato Caccioppoli, 2 – 84084 Fisciano (SA), Italy}
\email{lspada@unisa.it}

\maketitle
\begin{abstract}
Let $\sA$ be a homological category and $U\colon \sB\to \sA$ be a faithful conservative right adjoint. We introduce the notion of \emph{relative ideal} with respect to $U$ and show that, under suitable conditions,  any object of $\sA$ can be seen as a relative ideal of some object in $\sB$. We then develop a \emph{case study}: we first prove that the category of hoops is semi-abelian and that the category of MV-algebras is protomodular; then we apply our results to the forgetful functor from the category of MV-algebras to the category of Wajsberg hoops.
\end{abstract}


\section{Introduction}
The forgetful functor
\[
U\colon \CRing\to \CRng\,,
\]
from the protomodular category of commutative unital rings to the semi-abelian category of commutative rings has a left adjoint $F$ that freely adds the multiplicative identity. In greater detail, if $R$ is a commutative ring, then $F(R)$ has underlying abelian group given by the direct product $R\times \mathbb{Z}$, endowed with a multiplication defined by the formula
\begin{equation}\label{eq:multiplication}
(r,n)(r',n')\df (rn'+nr'+rr',nn') 	
\end{equation}
and with multiplicative identity  $(0,1)$. It is easy to see that the ring $R$ is contained in $F(R)$ as an ideal.  By the arbitrary choice of $R$, this means that every commutative ring can be seen as an ideal of a suitable \emph{unital} commutative ring. More precisely, one can prove that there is an equivalence of categories $\CRing/\mathbb{Z}\simeq \CRng$ which is defined by taking kernels in $\CRng$ of the ``objects'' in the slice category.  

Now, it is well known that ideals of commutative unital rings make it easier to deal with congruences (and quotients) in $\CRing$, but since they are not subobjects in $\CRing$, it is not immediate to describe  them categorically. On the other hand, ideals of commutative unital rings are subobjects in $\CRng$. This makes it possible to exploit the categorical properties of the semi-abelian category $\CRng$ in order to study the, still protomodular but not pointed, category $\CRing$.
 
The purpose of this article is to  set a study of these facts in a more general framework.  This is described at the beginning of Section \ref{sec:U-ideals}, where we introduce a notion of \emph{$U$-ideal}, with $U\colon \sB \to\sA$ a faithful conservative right adjoint functor, and $\sA$ a regular protomodular category. The varietal case is analyzed in detail in Theorem \ref{p:U-ideal=0-ideal}, where we prove that, under suitable conditions,  $U$-ideals are the same as the $0$-ideals due to Ursini \cite{U}. In Theorem \ref{thm:main} we establish a general version of the above equivalence, when the codomain $\sA$ is a homological category. 

In Section \ref{sec:MV}, we develop a case study. In Theorem \ref{prop:Hoops_is_semi-abelian}, we prove that the category $\Hoops$ is semi-abelian, so that our setting applies to certain categories of interest in algebraic logic, such as Wajsberg, Product, and G\"odel hoops, with respect to MV, Product and G\"odel algebras. We analyze in detail the case of the forgetful functor
\[U\colon \MVAlg\to \WHoops\]
from the category of MV-algebras into the category of Wajsberg hoops. 
In Corollary \ref{cor:MV} we establish an equivalence of categories $\MVAlg/\two\simeq\WHoops$ --- $\two$ being the initial MV-algebra. In particular, this shows that every Wajsberg hoop can be seen as (the order dual of) an ideal of a suitable MV-algebra.

\section{Preliminaries}
In this section, we recall some standard categorical and algebraic notions and we fix the related notation.  For the sake of the reader, we refer to \cite{Handbook2,BB} where all these notions can be found, rather than pointing to the original papers where the concepts were introduced. 

Let $\sB$ be a category. We denote by $\Pt_\sB$ the category whose objects (called \emph{points}) are the split epimorphisms of $\sB$, together with a chosen section, i.e.\ the four-tuples $(B,J,p,s)$ in $\sB$ such that $p\colon B\to J$ and $p\circ s=id_J$. The morphisms of $\Pt_\sB$ are pairs of arrows $(f,g)\colon (B',J',p',s')\to (B,J,p,s)$ such that both the upward and downward directed squares below commute.
\begin{equation}\label{diag:split_epi_morphism}
\begin{aligned}\xymatrix{
B'\ar@<-.5ex>[d]_{p'}\ar[r]^{f}
&B\ar@<-.5ex>[d]_{p}
\\
J'\ar@<-.5ex>[u]_{s'}\ar[r]_g
&J\ar@<-.5ex>[u]_{s}}
\end{aligned}
\end{equation}   
Given an object $J$ of $\sB$, we denote by $\Pt_\sB(J)$  the subcategory of $\Pt_\sB$ obtained by fixing $J'=J$ and $g=id_J$ in diagram \eqref{diag:split_epi_morphism}. This is called \emph{the category of points over $J$}. In order to simplify the notation, a morphism between two points $(f,id_J)\colon (B',J,p',s')\to (B,J,p,s)$ will usually be written $f\colon (p',s')\to (p,s)$.

If the category $\sB$ has pullbacks, any arrow $g\colon J'\to J$ determines a functor $g^*\colon \Pt_{\sB}(J)\to  \Pt_{\sB}(J')$ defined by pulling back along $g$ the objects of $\Pt_{\sB}(J)$, as shown in the pullback diagram below:
\begin{equation}\label{diag:change_of_base}
\begin{aligned}\xymatrix{
J'\times_{J}B \ar@<-.5ex>[d]_{\pi_1}\ar[r]^-{\pi_2}
&B\ar@<-.5ex>[d]_{p}
\\
J'\ar@<-.5ex>[u]_{s'}\ar[r]_g
&J\ar@<-.5ex>[u]_{s}}
\end{aligned}
\end{equation}
Here $s'$ is the unique section of the first pullback projection $\pi_1$ obtained by the universal property of the pullback, i.e.\ it is such that $\pi_1\circ s'=id_{J'}$ and $\pi_2\circ s'=s\circ g$. The category $\sB$ is \emph{protomodular} \cite{B91} if, for every arrow $g$, the functor $g^*$ defined above is conservative, i.e.\ it reflects isomorphisms. 

A category is \emph{pointed} if it has a zero-object, that is, an object $0$ that is both terminal and initial.  In a pointed category, a \emph{kernel} of a map $p\colon B\to J$ is the pullback of $p$  along the initial map $0\to J$. More generally, a \emph{kernel pair} of $p$ is the pullback of $p$ with itself. An (internal) equivalence relation is called \emph{effective} when it is the kernel pair of a morphism. A category $\sB$ is \emph{regular}, if it is finitely complete, has pullback-stable regular epimorphisms, and all effective equivalence relations admit coequalizers. A regular category is \emph{Barr-exact} \cite{Barr} when all equivalence relations are effective. 
A pointed regular protomodular category is called \emph{homological}; it is called \emph{semi-abelian} \cite{JMT02} if it is also Barr exact and has finite coproducts. The notion of protomodular (and semi-abelian) category encompasses a wide range of categories of interest to algebraists, including categories of groups, rings, Lie algebras, associative algebras, and cocommutative Hopf algebras, among others. Of course, all abelian categories are semi-abelian. 

Finally, recall that, if a finitely complete category is pointed, protomodularity is equivalent to the so-called \emph{split short five lemma}. For the reader's convenience, here we recall a version of it. Let $(p,s)$ and $(p',s')$ be two split epimorphisms  over $J$, with kernels $(K,k)$ and $(K,k')$ respectively, 
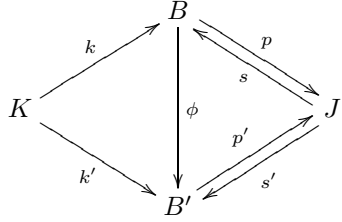
\[
\xymatrix@C=10ex{
&B\ar@<+.5ex>[dr]^p\ar[dd]^\phi
\\
K\ar[ur]^k\ar[dr]_{k'}&&J\ar@<+.5ex>[ul]^s\ar@<+.5ex>[dl]^{s'}
\\
&B'\ar@<+.5ex>[ur]^{p'}}
\]
if  $\phi$ is a morphism between them, i.e.\ it is such that $p'\circ \phi=p$, $\phi\circ s=s'$ and $\phi\circ k =k'$, then $\phi$ is an isomorphism.

\section{Relative \texorpdfstring{$U$}{U}-ideals}\label{sec:U-ideals}
In this section we provide a definition that comprises the leading example in the Introduction and several more.
\subsection{Basic setting for relative \texorpdfstring{$U$}{U}-ideals}
Let $\sA$ be a pointed category with pullbacks and $U\colon\sB \to \sA$ be a faithful functor.  \begin{definition}\label{d:ideal}
A \emph{relative  $U$-ideal} of an object $B$ of $\sB$ (or $U$-ideal, or simply ideal, when $U$ is understood) is a morphism $k\colon A\to U(B)$  of $\sA$ such that there exists a morphism $f\colon B\to B'$ of $\sB$ that makes the following diagram a pullback in $\sA$:
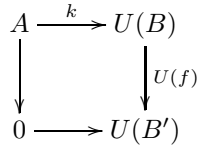
\[
\xymatrix{
A\ar[r]^-k\ar[d]&U(B)\ar[d]^{U(f)}\\
0\ar[r]&U(B')
}
\]
\end{definition}
In other words, a $U$-ideal is a kernel in $\sA$ of a map that lives in $\sB$. Since kernels are monomorphisms, $U$-ideals can be seen as subobjects in a ``larger'' category; when the monomorphism $k$ is understood, we feel free to say that the object $A$ is a $U$-ideal of $B$. Prototypical examples of $U$-ideals come from the inclusion $U\colon \Ring\to\Rng$ of unital rings into rings: $U$-ideals in $\Ring$ are nothing but the usual bilateral ideals of ring theory. Notice that in concrete situations, namely, when $U$ is a forgetful functor between concrete categories, we identify the subobject $A$ of $U(B)$ with its image, namely a subset $A \subseteq B$.

The example from ring theory suggests considering a more robust environment for dealing with the notion of $U$-ideal. 
\begin{definition}\label{def:basic}
A \emph{basic setting} for relative $U$-ideals is an adjunction 
\begin{equation}\label{diag:adjunction}
\xymatrix{\sB\ar[r]_-{U}&\sA\ar@/_3ex/[l]_{F}^{\bot}},	
\end{equation}
 where the category $\sA$ is homological  and $U$ is a conservative faithful functor.  
\end{definition}
\begin{remark}
The right adjoint $U$  preserves finite limits and is conservative; therefore, it reflects protomodularity (see \cite[Example 3.1.9]{BB}); it follows that the category $\sB$ is protomodular, too.  On the other hand, the left adjoint  $F$ preserves finite colimits, so the object $I\df F(0)$ is initial in $\sB$.
 \end{remark}
\begin{remark}\label{rem:1}
	It is worth to observe that, most often, the requirement for $U$ to be faithful comes for free. More precisely, for a functor $U$ that preserves finite limits between finitely complete categories, it is easy to show that ``$U$ conservative'' implies  ``$U$ faithful''. 
\end{remark}

\subsection{Relative \texorpdfstring{$U$}{U}-ideals and varieties of algebras}
Let us recall the characterization of protomodular varieties, as  established by Bourn and Janelidze in \cite{BJ}.
\begin{theorem}\label{p:protomodular}
A variety $\sV$ of universal algebras is protomodular if and only if there is a natural number $n$, $0$-ary terms $e_1,\dots,e_n$, binary terms $\alpha_1,\dots,\alpha_n$, and $(n+1)$-ary term $\theta$ such that:
\begin{equation}
\begin{aligned}\label{eq:BJ}
	\sV&\models \theta(\alpha_1(x,y),\dots,\alpha_n(x,y),y)= x & \text{and}\\ 
	\sV&\models\alpha_i(x,x)= e_i & \text{for}\ i=1,\dots,n\,.
\end{aligned}
\end{equation}
\end{theorem}

\begin{remark}\label{r:proto-zero}
Since any variety is co-complete and Barr-exact, if $\sV$ is protomodular and pointed then it is semi-abelian.  Notice that being pointed is equivalent to having a unique definable constant $0$ \cite[Proposition 0.2.6]{BB}, so, in the above characterization, one has $n> 0$ and necessarily $e_1=\dots=e_n=0$. 
\end{remark}

 \subsection{Basic setting for varieties}  Let $\sA$ (resp.\ $\sB$) be the variety of universal algebras with signature $\Sigma_\sA$ (resp.\ $\Sigma_\sB$) defined by a set $Z_\sA$ (resp.\   $Z_\sB$) of equations. Moreover, we require $\sA$ to be homological (thus  semi-abelian),  $\Sigma_\sA\subseteq\Sigma_\sB$ and $Z_\sA\subseteq Z_\sB$.  The obvious forgetful functor $U\colon \sB\to \sA$ is an algebraic functor between the categories of set-theoretical models of two algebraic theories. 
We call a functor $U$ as above a \emph{basic setting for varieties}. Notice that, according to \cite[Proposition 3.5.7 and Theorem 3.7.7]{Handbook2}, $U$ is a faithful conservative right adjoint, thus satisfies the conditions of Definition \ref{def:basic}. 
\begin{remark}
	 Janelidze recently introduced the novel notion of {\em ideally exact category}, as a first step towards ``a development of a new non-pointed counterpart of semi-abelian categorical algebra'' (\cite{J23}). It is immediate to see that, in our \emph{basic setting for varieties}, the category $\sB$ is ideally exact (see \cite[Theorem~3.1, Definition~3.2]{J23}).
\end{remark}

We now recall the following definitions introduced by Ursini in \cite{U} (see also \cite{GU}).
 \begin{definition}
Let $\sV$ be a variety with a constant symbol $0$ in its signature $\Sigma_\sV$.
 \begin{itemize}
 	\item A term $t(x_1,\dots,x_m,y_1,\dots,y_n)$ over $\Sigma_V$ is called $0$-\emph{ideal term} in the variables $y_1,\dots,y_n$ if 
 \[
 \sV \models t(x_1,\dots,x_m,0,\dots,0)=0\,.
 \]
 	\item A non-empty subset $H$ of an algebra $A$ in $\sV$ is called  $0$-\emph{ideal} if, for every 0-ideal term  $t(x_1,\dots,x_m,y_1,\dots,y_n)$ in the variables $y_1,\dots,y_n$, any $a_1,\dots,a_m\in A$, and any $h_1,\dots,h_n\in H$, one has 
 \[
 t(a_1,\dots,a_m,h_1,\dots,h_n)\in H\,.
 \]
\end{itemize}
 \end{definition}
It is easy to see that the equivalence class of $0$ under any given congruence is a  0-ideal.  Vice versa, one calls $0$-\emph{ideal determined}, or just \emph{ideal determined}  (cf.\ \cite[Definition 1.3]{GU}) a  variety $\sV$ where every $0$-ideal is the equivalence class of $0$ for a unique congruence relation. A relevant subclass of the class of ideal determined varieties is the class of the so-called \emph{classically ideal determined} \cite{U3}. A variety $\sV$ is named so, when it satisfies the conditions of Theorem \ref{p:protomodular},  with all $e_i$'s equal to a single constant $0\in \Sigma_\sV$. By Remark \ref{r:proto-zero}, it is immediately noticed that all semiabelian varieties are classically ideal determined. Classically ideal determined varieties were introduced in \cite{U2} with the name \emph{BIT speciale} varieties, whereas ideal determined ones where called \emph{BIT} varieties in \cite{U}. The following result needs no explanation.
\begin{lemma}
	Let  $U\colon \sB\to \sA$ be a basic setting for varieties. Then $\sB$ is classically ideal determined.
\end{lemma}
The next result establishes a connection between the varietal notion of $0$-ideal and the categorical notion of $U$-ideal.
\begin{theorem}\label{p:U-ideal=0-ideal}
	Let  $U\colon \sB\to \sA$ be a basic setting for varieties.  A subset $H$ of an algebra $B$ of $\sB$ is a $0$-ideal of $B$ if and only if $H\subseteq  U(B)$ is a $U$-ideal of $B$ with respect to $U\colon \sB\to \sA$.
\end{theorem}
 \begin{proof}
Let $H$ be a $0$-ideal of the algebra $B$ of $\sB$. All the $0$-ideal terms of the variety $\sA$ are still $0$-ideal terms when considered in the variety $\sB$. Therefore, $H\subseteq U(B)$ is a $0$-ideal in $\sA$. On the other hand, by \cite[Equation 8.5]{JMU07} (see also \cite{MM10}), in any semi-abelian variety, the $0$-ideals coincide with kernels, thus $H$ is a kernel in $\sA$. Now we claim that it is a kernel of a homomorphism $U(f)$, with $f$ in $\sB$. Indeed, since $\sB$ is ideal determined, there exists a unique congruence $R$ on $B$ such that $H$ is the $0$-class of $R$. Let $f\colon B\to Q$ be the canonical quotient over $R$. By \cite[Proposition 3.7.5]{Handbook2}, $U$ preserves (and reflects) coequalizers of equivalence relations, thus $U(f)$ is the canonical quotient over $U(R)$   in $\sA$. It is easy to see that $H\subseteq U(B)$ is then the $0$-class of $U(R)$, thus its normalization. In conclusion, $H\subseteq U(B)$ is the kernel of $U(f)$, i.e.\ $H$ is a $U$-ideal.
 
 Vice versa, let $H\subseteq B$ be a $U$-ideal of $B$. This means that $H\subset U(B)$ is the kernel in $\sA$ of a morphism $U(f)$, i.e.\ it is the $0$-class of the kernel pair relation of $U(f)$. Now, again by  \cite[Proposition 3.7.5]{Handbook2}, $U$ preserves and reflects small limits, as a consequence the kernel pair of $U(f)$ is the image under $U$ of the kernel pair of $f$. Since $U$ is constant on the underlying sets, $H$ is the $0$-class of the corresponding congruence in $\sB$, hence it is a $0$-ideal of $B$.
 \end{proof}

\subsection{Augmentation \texorpdfstring{$U$}{U}-ideals}
 Let us come back to a more general (not necessarily varietal) situation and consider a basic setting as in \eqref{diag:adjunction}.  For every object $A$ of $\sA$, let $p_A$ be the  morphism of $\sB$  given by the universal property of the unit $\eta$ of the adjunction,  with respect to the zero morphism $0\colon A\to U(I) $. Namely, the unique map $p_A\colon F(A)\to I$ such that $U(p_A)$ makes the  triangle below commute:
\begin{equation}\label{diag:unit}
\begin{aligned}
\xymatrix{F(A)\ar[d]_{\exists! p_A}^{\ \ \ \text{such that} }& A\ar[r]^-{\eta_A}\ar[dr]_-{0}&UF(A)\ar@{-->}[d]^{U(p_A)}\\I&&U(I)	}
\end{aligned}
\end{equation}
\begin{definition}\label{d:augmentation}
	We say that $\eta_A$  is an \emph{augmentation $U$-ideal}, or more simply an \emph{augmentation ideal}, if it is the kernel of $U(p_A)$. \end{definition}
Notice that the proposed terminology is consistent, since any augmentation $U$-ideal is a $U$-ideal of $F(A)$ according to Definition \ref{d:ideal}.

\begin{remark}\label{rem:2}
Recall from \cite[Proposition 2.1]{T91} that the unit $\eta$ of an adjunction is called \emph{cartesian}  when the squares expressing its naturality are pullbacks. For a basic setting like in \eqref{diag:adjunction}, there is an unexpected connection between the notion of augmentation ideal and that of  cartesian unit: the unit $\eta$ of the adjunction is cartesian if and only if 
\begin{equation}\tag{$*$}\label{condition}
\text{for every}\ A\ \text{in}\ \sA,\ \eta_A\ \text{is an augmentation ideal.} 
\end{equation}
\end{remark}

\begin{theorem}\label{thm:main} 
Suppose that in a basic setting like in \eqref{diag:adjunction}  condition {\em (\ref{condition}) } holds. Then, the kernel functor 
	\[
	K\colon \sB/I\to \sA
	\]
	defined on objects by letting $K(f)\df \Ker (U(f))$,
	is an equivalence of categories.  (Notice that the kernel is computed in $\sA$.)
\end{theorem}
\begin{proof}
We observe that, since $I$ is initial in $\sB$, every object in $\sB/I$ is a split epimorphism whose unique section is given by the initial arrow. Thus, the slice category $\sB/I$ is the same as the category $\Pt_{\sB}(I)$ of points over $I$, and one can consider the functor $K\colon \Pt_{\sB}(I)\to \sA$, defined on objects by 
\[
K(f,s)\df \Ker (U(f))\,.
\]

\underline{$K$ is essentially surjective.} Indeed, given any object $A$ of $\sA$, let us consider the point $(p_A,\sigma_A)$, where $p_A$ is the map of diagram \eqref{diag:unit} and $\sigma_A$ its unique section.  $K(p_A,\sigma_A)$ is a kernel of $U(p_A)$, moreover, by condition \eqref{condition}, also $A$ is a kernel of $U(p_A)$, therefore  the two are isomorphic to each other.

\underline{$K$ is faithful.} Indeed, let us consider two morphisms $\phi,\phi'\colon (f,s)\to (f',s')$ between points over $I$, and suppose $K(\phi)=\psi=K(\phi')$. The situation is represented by the diagram below:
\[
\xymatrix@C=10ex{
\ar[dd]_{\psi}
\Ker(f)\ar[r]^-i&U(B)\ar@<+.5ex>[drr]^-{U(f)}\ar@<-.5ex>[dd]_{U(\phi)}\ar@<+.5ex>[dd]^{U(\phi')}
\\
&&&U(I)\ar@<+.5ex>[ull]^{U(s)}\ar@<+.5ex>[dll]^{U(s')}
\\
\Ker(f')\ar[r]_-{i'}&U(B')\ar@<+.5ex>[urr]^-{U(f')}
}
\]
Recall that, since $\sA$ is homological and hence protomodular, by \cite[Lemma 3.1.22]{BB} the pair $(i,U(s))$ is jointly strongly epic. We prove that $U(\phi)$ and $U(\phi')$ are equal by precomposing them with the pair $(i,U(s))$. Indeed, by hypothesis, 
\[
U(\phi)\circ i=i'\circ\psi=U(\phi')\circ i\,.
\]
Furthermore, since $I$ is initial,
\[
\phi \circ s=s'= \phi'\circ s\,,
\]
and therefore 
\[
U(\phi) \circ U(s)=U(\phi') \circ U(s)\,.
\]
Hence $U(\phi)=U(\phi')$, and by the faithfulness of $U$, we conclude $\phi=\phi'$.

\underline{$K$ is full.} We have to prove that, given two points $(f,s)$ and $(f',s')$ as above, and a morphism $\psi\colon K(f,s)\to K(f',s')$, one can find a morphism of points $\phi\colon (f,s)\to(f',s')$ such that $K(\phi)=\psi$.
To this end, let us consider the following construction:
\[
\xymatrix@C=10ex{
\ar[dd]_{\psi}\ar[dr]_{\eta_{\Ker(f)}}
\Ker(f)\ar[r]^-i&U(B)\ar@<+1ex>[drr]^-{U(f)}\\
&UF(\Ker(f))\ar@{-->}[u]^{U(\tau)}\ar@{-->}[d]_{U(\gamma)}
\ar@<+.5ex>[rr]^{U(p_{\Ker(f)})}
&&U(I)\ar@<+0ex>[ull]^(.7){U(s)}\ar@<+1ex>[dll]^{U(s')}\ar@<+.5ex>[ll]^{U(\sigma_{\Ker(f)})}
\\
\Ker(f')\ar[r]_-{i'}&U(B)'\ar@<+0ex>[urr]^(.3){U(f')}
}
\] 
By the universal property of the unit $\eta_{\Ker(f)}$, we obtain the unique morphism $\tau\colon F(\Ker(f)) \to B $ such that $U(\tau) \circ \eta_{\Ker(f)} =i$. In fact, one also has $f\circ \tau=p_{\Ker(f)}$, since, by the universal property of the unit, there exists a unique arrow $F(\Ker(f))\to I$ whose image under $U$, precomposed with $\eta_{\Ker(f)}$, equals $0$.
 Notice that, by the short five lemma, $U(\tau)$ is an isomorphism. Since $U$ is conservative, $\tau$ is an isomorphism too.
Similarly, by the universal property of the unit, there  exists a unique morphism  $\gamma\colon  F(\Ker(f))\to B'$ such that $U(\gamma)\circ \eta_{\Ker(f)}=i' \circ \psi$. Moreover, repeating the same argument as above, one easily sees that $\gamma$ gives a morphism of points $(p_{\Ker(f)},\sigma_{\Ker(f)})\to(f',s')$ over $I$. Let us define $\phi=\gamma\circ \tau^{-1}$. It is immediate to notice that $\phi$ yields a morphism of points over $I$, as required. Moreover,  one readily checks that $\gamma\circ \tau^{-1}\circ i = \gamma\circ \eta_{\Ker(f)}=i'\circ \psi$, so  that $K(\phi)=\psi$.
\end{proof}

\begin{remark}\label{rem:3}
	The proof of Theorem \ref{thm:main} can also be obtained by a less explicit and more general argument, as explained by G.\ Janelidze in \cite[Remark 2.7(b)]{J23}. 
\end{remark}

\begin{example}[Unitalization of associative algebras]
Let us fix a commutative unital ring $R$ and consider the forgetful functor
\[
U\colon \UAss\to \Ass
\]
from the category of unital associative $R$-algebras to the semi-abelian category of associative $R$-algebras. The functor $U$ has a right adjoint $F$ which freely adds the multiplicative identity. Given an associative algebra $A$, the unital associative algebra $F(A)$  has an underlying abelian group given by the direct product $A\times R$
endowed with a multiplication defined by
\begin{equation}\label{eq:multiplication2}
(a,r)(a',r')\df (r'a+ra'+aa',rr'),
\end{equation}
a scalar multiplication
\begin{equation}\label{eq:multiplication3}
r(a,r')\df (ra,rr'),
\end{equation}
and a multiplicative identity  $(0,1)$. The ring $R$ is the initial algebra in $\UAss$.  Since condition \eqref{condition} holds, by Theorem \ref{thm:main}, we get an equivalence of categories $\UAss/R\simeq \Ass$, it follows that the $U$-ideals here are precisely the bilateral ideals of unital associative algebras. 
\end{example}
\begin{remark}
The functor $F$ is not conservative, in general. A counterexample in the case of associative algebras is given in \cite{A}.
\end{remark}
\begin{example}[Unitalization of $C^*$-algebras] Let $\CStar$ denote the category of non-unital $C^*$-algebras, with ${}^*$-homomorphisms as arrows. $\CStar$ is a semi-abelian category (see \cite{GR}), moreover, if $\UCStar$ denotes the category of unital $C^*$-algebras, the forgetful functor $U\colon \UCStar\to \CStar$ admits a faithful conservative left adjoint $F$. For a $C^*$-algebra $A$, its unitalization $F(A)$ underlies the complex vector space $A\oplus\mathbb{C}$, with multiplication as in \eqref{eq:multiplication2}, unit $(0,1)$ and involution
\[
(a,z)^*\df (a^*,\overline z)
\]
where $\overline  z$ is the complex conjugate of $z$.  Given that condition \eqref{condition} holds, Theorem \ref{thm:main}  establishes an equivalence of categories $\UCStar/\mathbb C\simeq\CStar$.
\end{example}
\begin{example}
	Recall that a Boolean ring is a commutative ring $R$ where all elements are idempotent. Boolean algebras can be seen as Boolean rings with unit, with the operations translating in this way:
\[
x\vee y =x+xy+y,\quad x\wedge y=xy,\quad x+y=(x\wedge\neg y)\vee (\neg x\wedge y)\,.
\]
Let us denote by $U\colon \BooAlg\to \BooRng$  the forgetful  functor from Boolean algebras to Boolean rings, which  is nothing but the restriction to $\BooAlg$ of the functor $U\colon \CRing\to \CRng$ recalled in the introduction. In fact, the left adjoint restricts, too. Since the category of Boolean rings is semi-abelian and condition \eqref{condition} holds, Theorem \ref{thm:main} applies and provides the equivalence of categories  $\BooAlg/B_2\simeq\BooRng$, where $B_2=\{0,1\}$ is the initial Boolean algebra. 

The next section is devoted to the study of a similar situation, stemming from algebraic logic, that involves MV-algebras instead of Boolean algebras.
\end{example}

\section{Hoops and basic algebras}\label{sec:MV}
In this section, we prove that the variety of hoops is semi-abelian. From this fact, we immediately deduce that many varieties of algebras used in substructural logic are protomodular. In \cite{PtJ} Johnstone proves that the variety $\HSLat$ of Heyting semilattices is semi-abelian and, in turns, that the variety $\HAlg$  of Heyting algebras is protomodular. His argument uses the characterization of protomodular varieties by Bourn and Janelidze (as recalled in Theorem \ref{p:protomodular}). Our proof of the fact that hoops are semi-abelian relies on the same terms as those used in \cite{PtJ}, with only one of them adapted to our more general framework. Notice that our result is not a special case of Johnstone's; on the contrary, the latter can be deduced from our Theorem \ref{prop:Hoops_is_semi-abelian}. 

At the end of the section, we apply Theorem \ref{thm:main} to the forgetful functor from MV-algebras to Wajsberg hoops. 
We refer the interested reader to \cite{HB,BF00,BP94,CDM} for the missing notions on hoops and basic algebras.

\subsection{Hoops are semi-abelian}

A \emph{hoop} is an algebra $(A;\cdot,\to,1)$ such that $(A;\cdot,1)$ is a commutative monoid and the following equations hold:
\begin{enumerate}[(H1)]
\item\label{d:hoops:item2} $x\to x=1$\,;
\item\label{d:hoops:item3} $x\cdot(x\to y) =y\cdot(y\to x)$\,;
\item\label{d:hoops:item4} $(x\cdot y)\to z =x\to(y\to z)$\,.
\end{enumerate}
Obviously, the class of hoops forms an algebraic variety (see \cite[Theorem 1.8]{BP94}), which we denote by $\Hoops$.
In the following, we adopt the convention that the monoid operation $\cdot$ is more binding than the residual operation $\to$. Any hoop is partially ordered by the relation $\leq$ (often referred to as its \emph{natural order}), defined by
\begin{equation}\label{eq:hoops_order}
x\leq y \quad \text{iff}\quad x\to y=1 \quad \text{iff}\quad \exists u(x=u\cdot y)\,.
\end{equation}
In fact, the rightmost condition above describes the inverse right-divisibility relation of the monoid $(A;\cdot,1)$, where $u$ is the residual $y\to x$. Furthermore, with respect to this partial order, hoops are $\wedge$-semilattices and the meet is term-definable as \begin{equation}\label{eq:D}
\tag{D} x\wedge y\df x\cdot (x\to y).
\end{equation}
Hoops with their homomorphisms form a category which we denote by $\Hoops$.  

Given a hoop $(A;\cdot,\to,1)$,  a subset $F$ of $A$ is called an \emph{implicative filter}, or simply a \emph{filter},  if $1\in F$ and whenever $x\in F$ and $x\to y\in F$, then also $y\in F$. 
\begin{lemma}\textup{\mbox{\cite[Theorem 1.6]{BP94}}}
A subset $F\subseteq A$ is a filter on $A$ if and only if $(F;\cdot, 1)$ is a submonoid of $(A;\cdot, 1)$ which is upward closed with respect to the natural order $\leq$ of the hoop. 
\end{lemma}

In $\Hoops$, filters determine congruences and vice versa (see \cite[Theorem 1.11]{BP94}); in fact, the formers coincide with kernels of homomorphisms and are also subhoops.

\begin{theorem}\label{prop:Hoops_is_semi-abelian}
The variety $\Hoops$ is semi-abelian.
\end{theorem}

\begin{proof}
Since $\Hoops$ is a variety of algebras, it has finite coproducts and is Barr-exact. Furthermore, $\Hoops$ is pointed, as the one-element hoop is both initial and terminal.   To conclude, we prove that $\Hoops$ is protomodular.  Define the terms:
\[e_1\df 1\,,\qquad e_2\df 1\,, \qquad \alpha_1(x,y)\df  (x\to y)\,,\]
\[\qquad \alpha_2(x,y)\df (((x\to y)\to y)\to x)\,,\qquad\theta(x,y,z)\df ((x\to z)\cdot y)\,.\]
Now, $\alpha_1(x,x)= x\to  x=1$ by axiom \ref{d:hoops:item2};  similarly, 
\[\alpha_2(x,x)=((x\to x)\to x)\to x = (1\to x)\to x= x\to x =1,\]
where $1\to x=x$ by \cite[Lemma 1.5]{BP94}.  Finally,
\begin{align*}
\theta(\alpha_1(x,y),\alpha_2(x,y),y)&=  ((x\to y)\to y)\cdot (((x\to y)\to y)\to x) &
\\
&= x\cdot (x\to ((x\to y)\to y))&  \text{by \ref{d:hoops:item3}} 
\\
&= x\cdot (x\cdot (x\to y)\to y) & \text{by \ref{d:hoops:item4}} 
\\
&= x\cdot ((x\to y)\to(x\to y))& \text{by \ref{d:hoops:item4}}
\\
&= x\cdot 1 =x & \text{by \ref{d:hoops:item2}}
\end{align*}  
Thus, Theorem \ref{p:protomodular} applies.
\end{proof}
Notice that, as remarked in \cite{BJ}, $p(x,y,z)\df \theta(\alpha_1(x,y),\alpha_2(x,y),z)$ is a Mal'tsev term. Indeed, $p(x,y,y)=x$ holds by the above computations. Moreover:
\begin{align*}
p(x,x,y)=  (\alpha_1(x,x)\to  y)\cdot \alpha_2(x,x)=(1\to  y)\cdot1=y.
\end{align*}  
Thus, as any semi-abelian variety, $\Hoops$ is also a Mal'tsev variety.

We briefly recall below the definitions of some subvarieties of hoops well-known for their role in the algebraic study of logic. The equations on the left define the subvarieties of hoops indicated on the right, respectively. Below we use $x\vee y$ as an abbreviation for $((x\to y)\to y) \wedge ((y\to x)\to x)$.
\begin{align}
	\tag{W}\label{eq:W} &(x\to y)\to y=(y\to x)\to x & \text{\emph{Wajsberg hoops} (\WHoops)}\\
	\tag{S}\label{eq:S} &(x\to y)\vee (y\to x)=1	 & \text{\emph{Basic hoops} (\BHoops)}\\
	\tag{I}\label{eq:I} &x\cdot x=x 	 & \text{\emph{Idempotent hoops} (\IHoops)}\\
	\notag 		 &\eqref{eq:S} + \eqref{eq:I} 	 & \text{\emph{G\"odel hoops} (\GHoops)}\\
	\notag 		&\eqref{eq:S} +  (x\to z)\vee ((y\to x\cdot y)\to x)=1 	 & \text{\emph{Product hoops} (\PHoops)}.
\end{align}

Idempotent hoops are term equivalent to \emph{Brouwerian semilattices} (also known as \emph{Heyting semilattices}, i.e.\ the $\{\wedge,\to,1\}$-subreducts of Heyting algebras) by defining the hoop operation as $x\cdot y\df x\wedge y$ in any Brouwerian semilattice (see also \cite[Theorem 1.22]{BP94}).  
If the theory of Basic hoops is expanded by adding a constant $0$ and an axiom stating that $0$ is the minimum of the algebra, one obtains \emph{Basic algebras}.  The same expansion applied to Wajsberg, Product and G\"odel hoops, yields the varieties of \emph{Wajsberg}, \emph{Product} and \emph{G\"odel algebras}. Algebras obtained through the expansion of idempotent hoops are term equivalent to Heyting algebras (see \cite[Section 4]{AFM}).

Thus, Theorem \ref{prop:Hoops_is_semi-abelian} has two immediate corollaries. The first, to our knowledge, is new; the second is Johnstone's result in \cite{PtJ}.
\begin{corollary}\label{c:WH-semi-abelian}
The varieties  of Basic,  Wajsberg,  Product and G\"{o}del algebras are protomodular and the varieties of Basic,  Wajsberg,  Product and G\"{o}del hoops are semi-abelian.
\end{corollary}
\begin{corollary}
The variety $\HSLat$ of Heyting semilattices is semi-abelian, and the variety $\mathsf{IAlg}$ of Heyting algebras is protomodular.
\end{corollary}

Summing up, the algebraic categories $\GHoops$, $\PHoops$, $\BHoops$ and $\WHoops$ are semi-abelian by Corollary \ref{c:WH-semi-abelian}, thus homological. Clearly, forgetful functors $U\colon \mathsf{XAlg}\to \mathsf{XHoops}$ that forget the constant 0 are algebraic for any $\mathsf{X}\in \{\mathsf{G},\mathsf{P},\mathsf{W}, \mathsf{B}, \mathsf{I}\}$, and finally the two-element algebra $\two=\{0,1\}$ is initial in each of the above-mentioned categories of algebras. This provides several instances of the basic setting of diagram  \eqref{diag:adjunction}. In the next section, we explore the case of Wajsberg algebras.
Although we do not explore in detail the other cases apart from MV-algebras, we mention that the recent results in \cite{GMU} on the algebraic closure in Product and G\"{o}del hoops hint at some connection with our setting.

\subsection{Wajsberg hoops and MV-algebras}
\begin{definition}\label{d:MV-algebras}
An \emph{MV-algebra} is an algebra $(A;\oplus,\neg,0)$ that satisfies:
\begin{enumerate}
\item\label{d:MV-algebras:item-2} $x\oplus (y\oplus z)=(x\oplus y)\oplus z$\,;
\item\label{d:MV-algebras:item-1} $x\oplus y=y\oplus x$\,;
\item\label{d:MV-algebras:item0} $x\oplus 0=x$\,;
\item\label{d:MV-algebras:item1} $\neg\neg x=x$\,;
\item\label{d:MV-algebras:item2} $x\oplus \neg0=\neg0$\,;
\item\label{d:MV-algebras:item3} $\neg(\neg x\oplus y)\oplus y = \neg(\neg y\oplus x)\oplus x$\,.
\end{enumerate}
\end{definition} 
Further derived operations can be defined as follows:
\begin{align*}
&1\df \neg0\,,\qquad x\odot y\df \neg(\neg x\oplus\neg y)\,,\qquad x\rightarrow y\df \neg x \oplus y\\
&x\vee y\df  \neg(\neg x\oplus y)\oplus y \quad \text{ and } \quad x\wedge y\df  x\odot(\neg x \oplus y),
\end{align*}
It can be proved that with these definitions $(A;\odot,\rightarrow,1,0)$ is a Wajsberg algebra and  $(A;\vee, \wedge,0,1)$ is a bounded distributive lattice with top and bottom. In fact, as shown in \cite[Lemma 4.2.2 and Theorem 4.2.5]{CDM}, the variety of Wajsberg algebras is term equivalent to the variety $\MVAlg$ of MV-algebras. Therefore, modulo a term equivalence, there is a functor $U\colon \MVAlg\to \WHoops$ that forgets the constant $0$. 

The kernels of homomorphisms of MV-algebras are called ideals and can be characterized as follows. A subset $I\subseteq A$ of an MV-algebra is an \emph{ideal} if and only if it is a submonoid of $(A;\oplus,0)$ which is downward closed with respect to the natural order $\leq$ of the MV-algebra.
 This may cause terminology clash, since kernels of MV-algebras are no longer kernels when considered in the language of Wajsberg algebras. In fact, in MV-algebras, the roles of filters and ideals are completely interchangeable. Indeed, for any MV-algebra  $(A;\oplus,\neg,0)$, the negation operation yields an isomorphism of MV-algebras:
\begin{equation}\label{eq:neg_iso}
\neg\colon (A;\oplus,\neg,0)\longrightarrow  (A;\odot,\neg,1)
\end{equation}
that maps ideals to filters and vice versa. For this reason, we will (unconventionally) call \emph{kernels} the filters given by the inverse image of $1$ under an arbitrary MV-homomorphism. 

\subsection{The MV-closure revisited}
When the \emph{basic setting} is instantiated in the case MV-algebras, the left adjoint to $U\colon \MVAlg\to \WHoops$ is called \emph{MV-closure} of a Wajsberg hoop, and has been described in \cite{ACD}. It is obtained by freely adding a neutral element $0$ for the operation $\oplus$ defined in an arbitrary Wajsberg hoop $(W;\odot,\rightarrow,1)$ as
\[
w\oplus w'\df(w\rightarrow (w \odot w'))\rightarrow w'\,,\qquad w,w'\in W\,.
\]
The resulting MV-algebra has an operation $\oplus$ that extends the former.

Here we revisit the construction given in \cite{ACD}, adapting it to our notational choices.
Consider a Wajsberg hoop $(W;\odot,\rightarrow,1)$. We define the MV-algebra
\[
M(W)\df(W\times \two; \oplus,\neg, 0)\,,
\]
where $0\df(1,0)$, $\neg (w,i)\df (w,1-i)$, and 
\begin{align*}
&(w,1)\oplus(w',1)\df (w\oplus w',1)\,,\qquad (w,0)\oplus(w',0)\df (w\odot w',0)\,,\\
&(w,0)\oplus(w',1)=(w',1)\oplus (w,0)\df (w\rightarrow w',1)\,.
\end{align*}
This construction is functorial by \cite[Corollary 3.7]{ACD} and, as explained in \cite{ACD}, it gives the left adjoint $M$ to the forgetful functor $U$. 

As can be deduced from \cite[Theorem 3.5]{ACD}, for every Wajsberg hoop $W$, the unit component of the adjunction $\eta_{W}\colon W\to U(M(W))$ is given by $\eta_W(w)\df(w,1)$.  Furthermore, $\two$ is a retract of $M(W)$ by the maps $p_W(w,i)\df i$ and $\sigma_W(i)\df (1,i)$.  Therefore, there is a split short exact sequence:
\[
\xymatrix{
W\ar[r]^-{\eta_W}&UM(W)\ar@<+.5ex>[r]^-{U(p_W)}&U(\two)\ar@<+.5ex>[l]^-{U(\sigma_W)}}.
\]
Notice that the split epimorphism $p_W$ is a homomorphism of MV-algebras, therefore $\sigma_W$ is its unique section in $\MVAlg$. Moreover, the kernel of $p_W$ computed in $\WHoops$ is the unit $\eta_W$ of the adjunction, thus condition \eqref{condition} holds.
The following statement translates Theorem \ref{thm:main}.
\begin{corollary}\label{cor:MV}
	The kernel functor 
	\[
	K\colon \MVAlg  / \two \to \WHoops
	\]
	defined on objects by letting, for $f\colon A\to \two$, $K(f)\df\Ker (f)$,
is an equivalence of categories. 
\end{corollary}
As a consequence of Corollary \ref{cor:MV}, $W$ can be seen as a maximal filter of $M(W)$ (see also \cite[Theorem 3.5]{ACD}).

The inverse functor of the equivalence $K$ provides a \emph{unitalization functor}. 
As in the case of rings,  associative algebras, etc., the unitalization can be expressed as a semidirect product. A categorical reason for this is the fact that the slice category over an initial object coincides with the category of points and, in any semi-abelian category, points are equivalent to actions. In the case under investigation, this is better understood by making explicit the essential surjectivity of the inverse functor of $K$. To this end, let us consider a homomorphism of MV-algebras $f\colon A\to \two$. Since $\two$ is initial in  $\MVAlg$, $f$ comes equipped with a unique section $s\colon \two\to A$ which is an MV-homomorphism and witnesses the fact that $f$ is a retraction. Let us call $W$ the filter $f^{-1}(1)$, which is the kernel of $f$ with respect to its hoop structure $(\odot, \rightarrow,1)$. Then the assignment $\phi\colon M(W)\to A$ given by 
\begin{equation}\label{op:MV}
(w,i)\quad \mapsto \quad (s(i)\rightarrow w)\odot(w\rightarrow s(i))
\end{equation}
is an isomorphism, as we check by a case inspection.  Indeed,  $\phi$ preserves $\oplus$:
\begin{eqnarray*}
	\phi(w,1)\oplus\phi(w',1)&=& (1\rightarrow w)\odot(w\rightarrow 1)\oplus(1\rightarrow w')\odot(w'\rightarrow 1)\\
	&=& w\odot 1\oplus w'\odot 1\\
	&=& w\oplus w'\\
	&=& (1\rightarrow w\oplus w')\odot(w\oplus w'\rightarrow 1)\\
	&=&\phi(w\oplus w',1)\\
	&=&\phi((w,1)\oplus (w',1))\,,
\end{eqnarray*}
\begin{eqnarray*}
	\phi(w,0)\oplus\phi(w',1)&=& (0\rightarrow w)\odot(w\rightarrow 0)\oplus(1\rightarrow w')\odot(w'\rightarrow 1)\\
	&=& 1\odot \neg w\oplus w'\odot 1\\
	&=& \neg w\oplus w'\\
	&=& w\rightarrow w'\\
	&=&\phi(w\rightarrow  w',1)\\
	&=& \phi((w,0)\oplus(w',1))\,,
\end{eqnarray*}
\begin{eqnarray*}
	\phi(w,0)\oplus\phi(w',0)&=& (0\rightarrow w)\odot(w\rightarrow 0)\oplus(0\rightarrow w')\odot(w'\rightarrow 0)\\
	&=& 1\odot \neg w\oplus 1\odot \neg w'\\
	&=& \neg w\oplus \neg w'\\
	&=& \neg (w\odot  w')\\
	&=& (0 \rightarrow w\odot  w')\odot (w\odot  w'\rightarrow 0) \\
	&=&\phi(w\odot w' ,0)\\
	&=&\phi((w,0)\oplus(w',0))\,;
\end{eqnarray*}

\noindent $\phi$ preserves the negation:\\
\begin{minipage}{0.5\textwidth}
\begin{eqnarray*}
	\neg \phi(w,1)&=& \neg((1\rightarrow w)\odot(w\rightarrow 1))\\
	&=& \neg (w\odot 1)\\
	&=& \neg w\\
	&=& (0\rightarrow w)\odot(w\rightarrow 0) \\
	&=&\phi( w,0)\\
	&=& \phi(\neg (w,1))\,,
\end{eqnarray*}
\end{minipage}
\begin{minipage}{0.5\textwidth}
\begin{eqnarray*}
	\neg \phi(w,0)&=& \neg((0\rightarrow w)\odot(w\rightarrow 0))\\
	&=& \neg (1\odot \neg w)\\
	&=& w\\
	&=& (1\rightarrow w)\odot(w\rightarrow 1) \\
	&=&\phi( w,1)\\
	&=& \phi(\neg (w,0))\,.
\end{eqnarray*}
\end{minipage}\\

\noindent Finally, $\phi$ preserves the constant:
\[\phi(1,0)= (0\rightarrow 1)\odot(1\rightarrow 0)= 1\odot 0=0.\]

This shows that the map $\phi$ is a homomorphism of MV-algebras. Since $\phi$ commutes with the kernels, the sections and the retractions, is indeed an isomorphism, by the split short five lemma.

\section*{Acknowledgments}

The authors are grateful to George Janelidze for valuable comments that improved the first version of this paper, and  for suggesting  the observations gathered in Remark~\ref{rem:1} and Remark~\ref{rem:2}.

We acknowledge financial support under the National Recovery and Resilience Plan (NRRP), Mission 4, Component 2, Investment 1.1, Call for tender No. 1409 published on 14.9.2022 by the Italian Ministry of University and Research (MUR), funded by the European Union – NextGenerationEU– Project Title Quantum Models for Logic, Computation and Natural Processes (Qm4Np) – CUP F53D23011170001 - Grant Assignment Decree No. 1016 adopted on 07/07/2023 by the Italian Ministry of Ministry of University and Research (MUR).

The research was partially supported by GNSAGA - INdAM.

\end{document}